\documentclass[11pt]{amsart}
\usepackage{amsmath,amssymb,amscd}
\usepackage{amsmath,amssymb}

\usepackage{color}
\usepackage{amsxtra, amsmath}
\usepackage{amssymb, amscd}
\usepackage{graphicx}

\newcommand\NN{\mathbb N}

\newcommand\tr{\operatorname{tr}}

\theoremstyle{plain}
\newtheorem{thm}{Theorem}[section]

\newtheorem{prop}[thm]{Proposition}
\newtheorem{cor}[thm]{Corollary}

\theoremstyle{definition}

\theoremstyle{remark}

\title[Time and band limiting for matrix valued functions ]{\sc
Time and band limiting for matrix valued functions: \\ an integral and a commuting differential operator}

\author{Gr\"unbaum F. A., Pacharoni  I. and  Zurri\'an I.}
\address{Department of Mathematics, University of California, Berkeley
CA 94705}
\email{grunbaum@math.berkeley.edu}
\address{CIEM-FaMAF, Universidad Nacional de C\'ordoba, C\'ordoba~5000, Argentina}
\email{pacharon@famaf.unc.edu.ar}
\address{Facultad de Matematicas, Pontificia Universidad Cat\'olica de Chile, Santiago, Chile}
\email{zurrian@famaf.unc.edu.ar}

\date{\today}

\thanks{This research was supported in part by CONICET grant PIP
112-200801-01533, SeCyT-UNC and FONDECYT 3160646}



\begin{document}

\begin{abstract}

  The problem of recovering a signal of finite duration from a piece of its Fourier transform was solved at Bell Labs in the $1960$'s, by exploiting a ``miracle": a certain naturally appearing integral operator commutes with an explicit  differential one. Here we show that this same miracle holds in a matrix valued version of the same problem.

\end{abstract}

\maketitle

\section{Introduction}

A typical inverse problem with partial and noisy data can be formulated as follows:

\medskip
The unknown is a function $f(x)$ defined in physical space, and we may have some a priori  knowledge about $f(x)$. For instance $f$ could be supported in a compact part of physical space. The data is a piece
of its Fourier transform $\mathcal F(k)$ contaminated by noise. What is the best use one can make of this data to get a stable reconstruction  of $f(x)$?

\medskip
Many imaging problems fall under this umbrella, and maybe the question above was put in this precise form first by Claude Shannon while coming up with a mathematical foundation for Communication Theory back around 1950. His problem was addressed
in a remarkable series of papers by (different combinations of) D. Slepian, H. Landau and H. Pollak all working at Bell Labs where C. Shannon had done initial
work before leaving for MIT. See \cite{SLP1,SLP2,SLP3,SLP4,SLP5}.

\medskip

One of us run into a situation of this kind in dealing with ``limited angle tomography," see \cite{DG79,DG81,N01}. A detailed exploration of the corresponding ``singular value decomposition" for this problem, see \cite{D83,L86,G86}, made strong contact with the work of these workers at Bell Labs.

\

With this motivation at hand, we can give an account of what we do in this paper: we start with a (matrix valued) version of
a second order differential operator.  Here Shannon would have started with the (scalar valued) second derivative.
  Our harmonic analysis    in terms    of eigenfunctions of this operator is the analog of Fourier analysis in the classical case.

We then build the analog of the ``time-and-band limiting" integral  operator,  which we will denote by $S$. We then show that the same ``lucky accident"  found by the workers at Bell labs holds here too: we can exhibit a second order differential operator, denoted by $\widetilde{D}$, such that$$S \widetilde{D}  = \widetilde{D}  S.$$

This has, as in the original case of Shannon, very important numerical consequences: it gives a reliable way to compute the eigenvectors of $S$, something that cannot be done otherwise.

\medskip

The eigenfunctions of $S$ and $\widetilde {D}$ are the same ones, but  by using the differential operator instead of the integral one we have a manageable numerical problem: while the integral operator has a spectrum with eigenvalues that are extremely close together, the differential one has a very spread out spectrum, resulting in a stable numerical computation.

\medskip

Previous explorations of   a commutativity property  similar to the one above in the matrix valued case can be seen in \cite{GPZ15,CG15}, dealing with a full matrix and a narrow banded one.
 Our work here is related to that in \cite{GPZ15} but there are some important differences: instead of studying the commutant of a matrix $M$, which arises in the problem of time and band limiting, we look for a symmetric differential operators of order two commuting with an integral operator, which is the other operator of interest featuring in the time and band limiting problem. This extends the work started in the previous references and is much closer to the  early work at Bell Labs in the scalar case.  When they consider the unit circle as physical space, they also have to deal with two different situations, as we do.

\

The work \cite{GLP82} was written with no particular application in mind but some years later some applications were developed from the results thereof, see \cite{SD06,SDW06} and its references. We believe that the results obtained here, where one is dealing with matrix valued functions defined on spheres, will   open  the possibility  of new applications in the future,    most likely to an inverse problem involving tensor quantities. While we deal here with $2 \times 2$ matrix valued functions, the theory of matrix valued spherical functions yields situations with matrices of arbitrary size. For a sample of references on ``tensor tomography" see \cite{GRZAP99,PSU14} and references in these papers.

\smallskip
We hope that by placing these mathematical tools in front of the inverse problem community we may encourage people to use them in concrete imaging problems.

\section{Preliminaries}

We follow the same notation as  in \cite{GPZ15}, but we include some motivation here for benefit of the reader.
  Let $W(x)$ be a    matrix weight function    in the open interval $(a,b)$ and let  \{$Q_w(x)\}_{w\in\NN_0}$
 be a sequence of  matrix orthonormal polynomials with respect to the weight $W(x)$.

The Hilbert spaces $\ell^2(M_R, \NN_0)$ and  $L^2((a,b), W(t)dt)$ are given by the real valued $R\times R$ matrix sequences
$(C_w)_{w\in \NN_0}$ such that $\sum_{w=0}^\infty \tr \left( C_w\,C_w^*\right) < \infty$  and all measurable matrix valued functions $f(x)$, $x\in (a,b)$, satisfying $\int_a^b \tr\left(f(x)W(x)f^*(x)\right)dx < \infty $, respectively.
A natural   analog of the Fourier transform is the isometry $F:\ell^2(M_R,\NN_0) \longrightarrow L^2(W)$ given by
$$(C_w)_{w=0}^\infty \longmapsto  \sum_{w=0}^\infty C_w Q_w(x).$$
In our case, the matrix-polynomials are dense in $L^2(W)$, then this map is unitary with the inverse $F^{-1}: L^2(W)\longrightarrow \ell^2(M_R,\NN_0) $ given by
$$ f \longmapsto C_w=\int_a^b f(x)\,W(x)\, Q^*_w(x) dx.$$

 If we consider the problem of determining a function $f$,  from the following (typically noisy) data: $f$ has support on the compact set $[0,N]$ and its Fourier transform $Ff$ is known on a compact set $[a,\alpha]$, one concludes that
we need to compute the singular vectors (and values) of the operator $E:\ell^2(M_R,\NN_0)\longrightarrow L^2(W) $ given by
$$E f= \chi_\alpha F \chi_N f,$$
where $\chi_N$ is the {\em time limiting  operator} and  $\chi_\alpha$ is the {\em band limiting operator}.

At level $N$, $\chi_N$ acts on $\ell^2(M_R,\NN_0)$ by simply setting equal to zero all the components with index larger than $N$. At level $\alpha$, $\chi_\alpha$ acts on $L^2(W)$ by multiplication by the characteristic function of the interval $(a, \alpha)$, $a<\alpha\le b$.

 This leads us to study the eigenvectors of the operators
$$E^*E= \chi_N F^{-1} \chi_\alpha F \chi_N\qquad \text{ and } \qquad E E^*= \chi_\alpha F \chi_N F^{-1} \chi_\alpha.$$
The operator $E^*E$, acting on $\ell^2(M_R,\NN_0)$ is just a finite dimensional block-matrix $M$, and each block is given by
$$(M)_{m,n}=(E^*E)_{m,n}= \int_a^\alpha Q_m(x) W(x) Q^{*}_n(x)  dx, \qquad 0\leq m,n \leq N.$$
The second operator $S= E E^*$ acts on $L^2((a,\alpha), W(t)dt)$ by means of the integral kernel
\begin{equation}\label{kernel}
  k(x,y)=\sum_{w=0}^N Q_w^*(x)Q_w(y).
\end{equation}
 The action of $k(x,y)$ is spelled out later in formula \eqref{intoper} more explicitly. 

\medskip

For general $N$ and $\alpha$ there is no hope of finding the eigenfunctions of $EE^*$ and $E^*E$ analytically. However, there is a strategy to solve this typical inverse problem:
  finding an operator with simple spectrum which would have the same eigenfunctions as the operators $E E^*$ or $E^* E$. This is exactly what Slepian, Landau and Pollack did in the scalar case, when dealing with the unit circle and the usual Fourier analysis. They discovered the following properties:
\begin{itemize}
  \item For each $N$, $\alpha$ there exists a symmetric tridiagonal matrix $L$, with simple spectrum, commuting with $M$.
  \item For each $N$, $\alpha$ there exists a self-adjoint differential operator $D$, with simple spectrum, commuting with the integral operator $S=EE^*$.
\end{itemize}

More than this is true: thanks to this lucky accident one replaces a costly ill-conditioned problem by a manageable well-conditioned one. It is hard to ask for a better situation.

\medskip

In \cite{GPZ15} we have dealt with the analog of the first property for the case discussed in \cite{PZ16}. In this paper we address the second one of the properties above in the same situation.

In both cases, the role of  the unit circle will be taken up by the $n$-dimensional sphere. We consider $2 \times 2$ matrix valued functions defined on the sphere with the appropriate invariance that makes them functions of the colatitude $\theta$ and we use $x=\cos(\theta)$ as the variable.
The role of the Fourier transform is taken by the expansion of our functions in terms of a basis of matrix valued orthogonal polynomials described below. This is similar to the situation discussed in \cite{GLP82} except for the crucial fact that our functions are now matrix valued.

\

\section{Integral and Differential operators}

Given the sequence of matrix orthonormal  polynomials $\{Q_w\}_{w\geq 0}$ with respect to the weight $W$, we fix a natural number $N$ and $\alpha \in (-1,1)$. We consider the integral operator $S$ with kernel $k$, defined in \eqref{kernel}, acting on $L^2((-1,\alpha), W)$ ``from the right hand side'':
\begin{equation}\label{intoper}
   (fS)(x)=\int_{-1}^\alpha f(y)W(y)\big(k(x,y)\big)^*dy
\end{equation}

 The restriction to the interval $[-1,\alpha]$ implements ``band-limiting" while the
restriction to the range $0, 1, . . . , N$ takes care of ``time-limiting.'' In the language of \cite{GLP82} where we were dealing with scalar valued functions defined on spheres the first restriction gives us a ``spherical cap" while the second one amounts to truncating the expansion in spherical harmonics.


The aim of this paper is to prove that there exists a symmetric differential operator $\tilde D$, defined in $[-1,\alpha]$,
commuting with the integral operator $S$, that is
$$\tilde D \,S=S\tilde D.$$





 Symmetry for the differential operator $\tilde D$ means  that
$$\langle P\tilde D, Q\rangle_\alpha=\langle P, Q\tilde D \rangle_\alpha, $$
 for an appropriate dense set of functions $P,Q$   where
\begin{equation}\label{W-trunc}
  \langle P,Q\rangle_\alpha=\int_{-1}^\alpha P(x)W(x)Q^*(x)\, dx.
\end{equation}

Notice that in principle there is no guarantee that we will find any such $\tilde D$ except for a scalar multiple of the identity. For the problem at hand we need to exhibit a differential operator  $\tilde D$ that has a simple spectrum, which would imply that its eigenfunctions are also eigenfunctions of the integral operator $S$ (see comments in Section \ref{CyO}).

\begin{prop}\label{Dk}
  Let $C$ be a symmetric differential operator and $S$  be an  integral operator with kernel $k$. Then
  $$C S=SC\qquad \text{ if and only if } \qquad
   \left( k(x,y)^*\right)C_x= (k(x,y)C_y)^*.$$
($C_x$  is meant to emphasize that $C$ acts on the variable $x$).
\end{prop}

\begin{proof}
  Let us observe that from the symmetry of $C$ we have
  $$ ((fC )S)(x)= \int_{-1}^\alpha (fC_y )(y)W(y)\big(k(x,y)\big)^*dy= \int_{-1}^\alpha  f(y)W(y)\big( k(x,y)C_y\big)^*dy. $$
  On the other hand
  $$((fS)C)(x)=\int_{-1}^\alpha f(y)W(y) \big(k(x,y)^*\big)C_xdy.$$

Therefore the proposition follows easily.
\end{proof}

\medskip

The polynomials considered here are those studied in \cite{PZ16}, given by the matrix-valued spherical functions associated with the $n$-dimensional sphere
$S^n\simeq G/K$, with $(G,K)=(\mathrm{SO}(n+1), \mathrm{SO}(n))$, studied in \cite{TZ14b}.
These spherical functions give rise to sequences $\{R_w\}_{w\in\NN_0}$ of monic matrix orthogonal polynomials  depending on two parameters $n$ and $p$ in $\mathbb R$ such that $0< p< n$. Namely,
\begin{equation*} \label{Pwdef}
R_w(x)=  \frac {w!\,(n+1)}{2^w (\tfrac{n+1}2)_w}\,\begin{pmatrix}
\frac{1}{n+1}\, C_w^{\frac{n+1}{2}}(x)+\frac{1}{p+w}\,C_{w-2}^{\frac{n+3}{2}}(x)&\frac{1}{p+w}\,C_{w-1}^{\frac{n+3}{2}}(x)\\ \mbox{} \\
\frac{1}{n-p+w}\,C_{w-1}^{\frac{n+3}{2}}(x)&\frac{1}{n+1}\, C_w^{\frac{n+1}{2}}(x)+\frac{1}{n-p+w}\,C_{w-2}^{\frac{n+3}{2}}(x)
\end{pmatrix},
\end{equation*}

\smallskip
\noindent where   $C_w^\lambda(x)$ denotes the $w$-th Gegenbauer polynomial
%
%
%
%
%
%
The matrix polynomials $\{R_w\}_{w\geq 0} $ are orthogonal with respect to the weight matrix
\begin{equation}\label{peso-x}
  W(x)=W_{p,n}(x)= (1-x^2)^{\tfrac n2 -1} \begin{pmatrix}
  p\,x^2+n-p & -nx\\ -nx & (n-p)x^2+p
\end{pmatrix},\qquad x\in [-1,1].
\end{equation}

\

The kernel $k(x,y)$ appearing in the definition of the integral operator $S$ is given    in terms    of  
the orthonormal sequence of matrix polynomials
\begin{equation}\label{relationQP}
  Q_w= S_wR_w,
\end{equation}
where $S_w=\|R_w\|^{-1}$ is the inverse of the matrix norm of $R_w$ given by
\begin{equation}\label{norma2Pw}
\begin{split}
\langle R_w,R_w\rangle  =\| R_w\|^2
  =
\frac{\sqrt \pi \;\Gamma(\tfrac n2 +1)(n+1)_w}{w!\,(n+1)(n+2w+1)\Gamma(\tfrac n2+\tfrac 32)} 
\begin{pmatrix}
   \frac {p\,(n-p+w+1)}{p+w}&0\\0& \frac{(n-p)(p+w+1)}{n-p+w}
\end{pmatrix}.
\end{split}
 \end{equation}

\

\begin{thm} \label{operatorD} (Theorem 3.1, \cite{PZ16}.) For each $w\in \NN_0$, the matrix  polynomial $R_w$ satisfies
$R_wD= \Lambda_w R_w(x), $ where
$$D= \frac{d^2}{dx} \,(1-x^2)-\frac{d}{dx} \,\Big( (n+2)x+2\left(\begin{smallmatrix}
  0&1\\1&0\end{smallmatrix}\right) \Big)-\left(\begin{smallmatrix}
    p&0\\0&n-p  \end{smallmatrix}\right), $$
and the eigenvalue   is given by
$$\Lambda_w(D)= \begin{pmatrix}
  -w(w+n+1)-p & 0\\ 0& -w(w+n+1)-n+p
\end{pmatrix}.$$
\end{thm}

\


 Let us remark that this differential operator acts on the variable $x$ ``from the right hand side."

\bigskip

Let us introduce the following right hand side differential operator $\tilde D$, acting on nice functions in the interval $(-1,\alpha)$.

\begin{equation} \label{opDtilde}
\begin{split}
  \tilde D= &\frac{d^2}{dx^2} \,(x^2-1)(x-\alpha)
+ \frac{d}{dx} \,\Big( (n+3) x^2 -\alpha(n+2) \, x -1+2 (x- \alpha) \left(\begin{smallmatrix}
  0&1\\1&0\end{smallmatrix}\right) \Big)\\
 & +\left(-N(N+n+2) x + \alpha (n-2p)
 \left(\begin{smallmatrix}   1&0\\0&0  \end{smallmatrix}\right)
    +     \left(\begin{smallmatrix}
    0&n-p+1\\p+1&0  \end{smallmatrix}\right)\right).
\end{split}
    \end{equation}
We will prove later  that $\tilde D$ is a symmetric operator and that it satisfies  $\left( k(x,y)^*\right)\tilde D_x= (k(x,y)\tilde D_y)^*$. Therefore $\tilde D$   commutes with the integral operator $S$ by Proposition \ref{Dk}.

\

Let us observe that the differential operator $\tilde D$ is somehow related to the differential operator $D$, given in Theorem \ref{operatorD}. Explicitly, we have
\begin{align*}
\tilde D=-D(x-\alpha)+ \frac{d}{dx}  (x^2 -1)+ \left(-N(N+n+2) - \left(\begin{smallmatrix}
    p&0\\0&n-p  \end{smallmatrix}\right)\right)x + \alpha (n-p)
    +
    \left(\begin{smallmatrix}
    0&n-p+1\\p+1&0  \end{smallmatrix}\right). \end{align*}

 This relation between the analogs of $D$ and $\tilde D$ is simpler in the case worked out at Bell Labs.

\noindent
For simplicity, we  use the notation
\begin{align*}
E_1= -N(N+n+2)\mathbf I- \left(\begin{smallmatrix}     p&0\\0&n-p  \end{smallmatrix}\right),
   &&
E_0=    \left(\begin{smallmatrix}     0&n-p+1\\p+1&0  \end{smallmatrix}\right).
\end{align*}

\

\begin{thm}
The differential operator $\tilde D$  is  symmetric with respect to the matrix valued inner product $\langle \, , \, \rangle_\alpha$ given in \eqref{W-trunc}.
\end{thm}
\begin{proof}
From \cite{GPT03} or \cite{DG04} we have that a differential operator $D=\frac{d^2}{dx^2} F_2(x)+\frac{d}{dx} F_1(x)+F_0$
is symmetric with respect to a weight $W$ defined in $(a,b)$ if and only if it satisfies, for $a<x<b$, the symmetry equations
\begin{equation} \label{symmeq}
  \begin{split}
  F_2 W & =WF_2^*,\\
   2(F_2W)'-F_1W &=WF_1^*,\\
 (F_2W)''-(F_1W)'+F_0W &=WF_0^*
  \end{split}
\end{equation}
and  the boundary conditions
\begin{equation}\label{boundary}
  \lim_{x\to a,b} F_2(x)W(x)=0, \quad \lim_{x\to a,b} \big (F_1(x)W(x)-WF_1^*(x)\big)=0.
\end{equation}

We have the following relations among the coefficients of the differential operators $D$ and $\tilde D$,
\begin{align*}
  \tilde F_2&= -(x-\alpha) F_2,\\
  \tilde F_1&= -(x-\alpha) F_1+(x^2-1)I,\\ 
  \tilde F_0&= -(x-\alpha) F_0+xE_1+E_0+\alpha I.
\end{align*}
We know that $D$ is a symmetric operator with respect to the weight $W$ for a subspace of functions defined in the interval $(-1,1)$, thus, the pair $\{W,D\}$ satisfies the equations \eqref{symmeq} and \eqref{boundary} with $a=-1, b=1$.
From this and by straightforward computations we verify that $\tilde D$  and $W$ also  satisfy the equations \eqref{symmeq} and the boundary conditions  \eqref{boundary} with $a=-1, b=\alpha$.
\end{proof}

\

 To prove that the differential operator $\tilde D$ commutes with the integral operator $S$ we have to verify that
 $\left( k(x,y)^*\right)\tilde D_x= (k(x,y)\tilde D_y)^*$, see Proposition \ref{Dk}. For this purpose, we will need  three different things: the explicit expressions of the three term recursion relation, the Christoffel-Darboux formula and a differentiation formula for a
matrix orthonormal sequence $\{Q_w\}_{w\ge 0}$. For the benefit of the reader we postpone the statements and  proofs to Section \ref{Secc-prop},  and proceed directly to the statement of our main result.

The three term recursion mentioned above can be thought of as a difference operator acting on the $w$ variable ``from the left hand side.''
We are in the presence of a ``bispectral" situation, and we will eventually see that it gives us the ``commuting miracle" mentioned at the beginning of this paper.

\

\begin{thm}  The differential operator $\tilde D$ satisfies
$$  \left( k(x,y)^*\right)\tilde D_x= (k(x,y)\tilde D_y)^*.$$
\end{thm}

\begin{proof}     Let $D$ the right hand side differential operator introduced in Theorem \ref{operatorD}. We have that the orthonormal polynomials
$\{Q_w\}_w$ are eigenfunctions of $D$ with the  eigenvalue $\Lambda_w$
, i.e. $Q_wD=\Lambda_w Q_w$.
 Then
\begin{align*}
  (k(x,y)\tilde D_y )^*
  &=-(y-\alpha)\sum_{w=0}^N Q_w^*(y)\Lambda_wQ_w(x)+(y^2-1) \sum_{w=0}^N \frac{d}{dy}Q_w^*(y) Q_w(x)\\
  & \quad + \sum_{w=0}^N \left( yE_1+E_0+\alpha(n-p)\right)^* Q_w^*(y)Q_w(x),
\end{align*}
and similarly
\begin{align*}
 \left( k(x,y)^*\right)\tilde D_x 
  &=-(x-\alpha)\sum_{w=0}^N Q_w^*(y)\Lambda_wQ_w(x)+(x^2-1) \sum_{w=0}^N \frac{d}{dx}Q_w^*(y) Q_w(x)
    \\ & \quad
  + \sum_{w=0}^N  Q_w^*(y)Q_w(x)\left( xE_1+E_0+\alpha(n-p)\right).
\end{align*}
%
%
Thus,
\begin{equation}
\begin{aligned}\label{a}
 ( k\tilde D_y)^*&-( k^*\tilde D_x)= \sum_{w=0}^N \left((x-y)Q_w^*(y)\Lambda_wQ_w(x) +(y^2-1)\frac{d}{dy}Q_w^*(y)Q_w(x) \right. \\ &  \left.-(x^2-1)Q_w^*(y)\frac{d}{dx}Q_w(x) +
 (E_0^*+yE_1^*)Q_w^*(y)Q_w(x)-Q_w^*(y)Q_w(x)(E_0+xE_1)\right).
\end{aligned}
\end{equation}

\noindent By using the differentiation formula given in Proposition \ref{diffON}, we have
\begin{equation}\label{c}
  \begin{split}
    ( k\tilde D_y)^*&-( k^*\tilde D_x)=\sum_{w=0}^N  (x-y)Q_w^*(y)\Lambda_wQ_w(x)  +yQ_w^*(y)(w-\overline F_w^*) Q_w(x)  \\
   &-xQ_w^*(y)(w-\overline F_w) Q_w(x) 
 +Q_w^*(y)(\overline G_w -\overline G_w^* )Q_w(x) \\ &  + \big(E_0+y( F_w+E_1)-\widetilde G_w\big)^* Q_w^*(y)Q_w(x)
- Q_w^*(y)Q_w(x)
\big(E_0+x( F_w+E_1)-\widetilde G_w\big)
\\ & \quad
-Q_{w-1}^*(y)\overline H_w^*Q_{w}(x)+ Q_{w}^*(y)\overline H_wQ_{w-1}(x).
  \end{split}
\end{equation}

If we take $$a_{21}=-1-\tfrac{n+2w}{(p+w)(n-p+w)}\, \quad \text{ and } \quad a_{22}=c_{22}=0,$$ in the matrices given in Propositions \ref{diff} and \ref{diffON}, we have
\begin{align*}
  F_w &= \begin{pmatrix}    p&0\\0&n-p   \end{pmatrix},  & G_w&= -\begin{pmatrix}    0&\frac{p(n-p+w+1)}{p+w} \\ \frac{(n-p)(p+w+1)}{n=p+w} &0   \end{pmatrix}, \\
 \widetilde G_w& = \begin{pmatrix}    0&n-p+1 \\ p+1 &0   \end{pmatrix},
& \;H_w&=(n+2w+n+1)A_w,
\end{align*}
where $A_w$ is the matrix in the three term recursion relation given in \eqref{ttrrmonic}.
In particular  we have  $$F_w+E_1= -N(N+n+2)\mathbf I  \quad \text{ and } \quad \widetilde G_w=E_0.$$
By using the explicit expression of $\|R_w\|$ (see \eqref{norma2Pw}), we get
$$\overline F_w=\|R_w\|^{-1} F_w\|R_w\|=F_w=\overline F_w^*, \qquad \overline G_w=\overline G_w^*, \qquad \overline H_w=(n+2w+n+1)\tilde A_w .$$

Therefore from \eqref{c} we obtain
\begin{equation}\label{b}
 \begin{aligned}
  ( k\tilde D_y)^*-( k^*\tilde D_x)= & \sum_{w=0}^N \, (x-y)Q_w^*(y)\Lambda_wQ_w(x)
 +(x-y)Q_w^*(y)(F_w-w) Q_w(x)
 \\
 &\quad 
+(x-y)N \left( N+n+2 \right) Q_w^*(y)Q_w(x) \\
& \quad  -
(n+2w+1)\left(Q_{w-1}^*(y)\tilde A_w^*Q_{w}(x)-
   Q_{w}^*(y)\tilde A_wQ_{w-1}(x)\right).
\end{aligned}
\end{equation}

\noindent By using the Christoffel-Darboux formula given in Proposition \ref{CrisDarboux},
\begin{equation*}
Q_{w-1}^*(y)\tilde A_w^*Q_{w}(x)-
   Q_{w}^*(y)\tilde A_wQ_{w-1}(x)=(x-y)
   \sum_{k=0}^{w-1}Q_{k}^*(y)Q_{k}(x),
\end{equation*}
and the fact that $\Lambda_w=\Lambda_w(D)=-w(w+n+1)-F_w$ (See Theorem \ref{operatorD}) we have
\begin{align*}
& \frac{ ( k\tilde D_y)^*-( k^*\tilde D_x)}{(x-y)}= \displaybreak[0]
\\& \quad =\sum_{w=0}^N\Big(Q_w^*(y)(-w(w+n+2)+N \left( N+n+2 \right))Q_w(x)  -
(n+2w+1) \sum_{k=0}^{w-1}Q_{k}^*(y)Q_{k}(x)\Big)  \displaybreak[0] \\
&\quad = \sum_{w=0}^N (-w(w+n+2)+N \left( N+n+2 \right)) Q_w^*(y)Q_w(x)  -
\sum_{w=0}^{N-1}\Big( \sum_{j=w+1}^N (n+2j+1)\Big) Q_{w}^*(y)Q_{w}(x)
\\ & \quad =0.
\end{align*}

\noindent Therefore we get $ ( k\tilde D_y)^*=( k^*\tilde D_x)$ and this  concludes the proof of the theorem.
\end{proof}

\

\begin{cor}
  The differential operator $\tilde D$ commutes with the integral operator $S$.
\end{cor}

\

\section{Properties of the   relevant orthogonal polynomials}\label{Secc-prop}

In this section we give some results about the sequence of matrix orthogonal polynomials with respect to the weight $W(x)$, introduced in \eqref{peso-x}. 
 Most of these results were used in the previous section to arrive at our main result.

\subsection{Three term recursion relation}
From the three term recursion relation for the sequence $\{P_w\}_{w\in\NN_0}$ given in \cite{PZ16} (Theorem 4.1), we get the
the three-term recursion relation for the monic orthogonal polynomials $\{R_w\}_{w\in\NN_0}$
\begin{equation}\label{ttrrmonic}
  x\,R_w(x)= A_w R_{w-1}(x)+B_wR_w(x)+R_{w+1}(x),
\end{equation}
where
\begin{align*}
  A_w& = \tfrac{w(n+w)}{(n+2w-1)(p+w)(n-p+w)(2w+n+1)}\begin{pmatrix}
  (p+w-1)(n-p+w+1){} &0\\0&  (p+w+1)(n-p+w-1)
\end{pmatrix},\\
 B_w &=\begin{pmatrix}
    0& \tfrac {-p}{(p+w)(p+w+1)}\\ \tfrac {-(n-p)}{(n-p+w)(n-p+w+1)}& 0
  \end{pmatrix}.
\end{align*}

\medskip

We have
\begin{align*}
 A_w = \|R_w\|^2\|R_{w-1}\|^{-2},&&
 (B_w\|R_w\|^2 )^* =B_w \|R_w\|^2.
\end{align*}
In fact, from the three term recursion relation \eqref{ttrrmonic} we get
\begin{align*}
  A_w\|R_{w-1}\|^2 & =\langle xR_w,R_{w-1}\rangle= \langle R_w,xR_{w-1}\rangle= \|R_w\|^2, \\
B_w\|R_{w}\|^2 &=\langle xR_w,R_{w}\rangle= \langle R_w,xR_{w}\rangle= \|R_w\|^2 B_w^*.
\end{align*}

\

\

An orthonormal sequence of matrix polynomials is given    in terms    of $R_w$ by
$  Q_w= S_wR_w$,
where $S_w=\|R_w\|^{-1}$ is the inverse of the matrix $\|R_w\|$.
Therefore the orthonormal polynomials $\{Q_w\}_{w\in\NN_0}$ satisfy the three-term recursion relation
$$ x\,Q_w(x)= \tilde A_w Q_{w-1}(x)+\tilde  B_w Q_w(x)+\tilde  A_{w+1}^* Q_{w+1}(x), $$
where
\begin{align*}
  \tilde A_w = S_wA_wS_{w-1}^{-1}=\|R_w\|\|R_{w-1}\|^{-1}
 && \text{and}&&
 \tilde B_w = \|R_w\|^{-1}B_w\|R_{w}\|.
\end{align*}
 We observe that $\tilde B_w^* =\tilde B_w$.

\medskip

\begin{prop}\label{CrisDarboux}
   The sequence of matrix orthonormal  polynomials $\{Q_w\}_{w\geq 0}$ satisfies the following Christoffel-Darboux formula
\begin{equation}
Q_{w-1}^*(y)\tilde A_w^*Q_{w}(x)-
   Q_{w}^*(y)\tilde A_wQ_{w-1}(x)=(x-y)
   \sum_{k=0}^{w-1}Q_{k}^*(y)Q_{k}(x).
\end{equation}
where
\begin{align*}
  \tilde A_w = S_wA_wS_{w-1}^{-1}=\|R_w\|\|R_{w-1}\|^{-1}
 && \text{and}&&
 \tilde B_w = \|R_w\|^{-1}B_w\|R_{w}\|.
\end{align*}
\end{prop}
\begin{proof}
 This result is proved in \cite{D96}.
\end{proof}

\

\subsection{Differentiation formulas} 

In this section we obtain several differentiation formulas for monic orthogonal polynomials.
There are four free parameters, namely $a_{21}$, $c_{12}$, $a_{11}$ and $a_{22}$.

\smallskip

\begin{prop}\label{diff} Let $\{R_w\} $ be the monic orthogonal polynomials associated to the weight $W=W_{p,n}$ introduced in \eqref{peso-x}. 
We have
  \begin{align*}
  (1-x^2)\frac {d R_w}{dx}(x)= & -w\, x R_w(x)+ x\big( F_w R_w(x)-R_w(x) F_w\big) +G_w R_w(x)\\
  &+ R_w(x) \widetilde G_w + H_wR_{w-1}(x),
  \end{align*}
 where
 \begin{align*}
   F_w&= -\tfrac{(n+2w)}{(p+w)(n-p+w)}\begin{pmatrix}     p&0\\0&n-p   \end{pmatrix}
   - a_{21}\begin{pmatrix}     p&0\\0&n-p   \end{pmatrix}
   +c_{12} \frac{(p+w)(n-p+w) }{(n-2p)}\begin{pmatrix}   0& p \\ n-p & 0   \end{pmatrix} +a_{11}\mathrm{Id},
   \\ \mbox{} \displaybreak[0] \\
  G_w&= \begin{pmatrix}     0&\frac{p(n-p+w)}{(p+w)^2} \\ \frac{(n-p)(p+w)}{(n-p+w)^2}&0   \end{pmatrix}
        + a_{21} \begin{pmatrix}   0& \frac{p(n-p+w)}{(p+w)}\\ \frac{(n-p)(p+w)}{(n-p+w)} & 0         \end{pmatrix}
        \\ &\qquad
        +c_{12} \big(w(w+n)-p(n-p)\big) \begin{pmatrix}
          1&0\\0&0         \end{pmatrix} +a_{22}\mathrm{Id},
    \\ \mbox{} \displaybreak[0] \\
 \widetilde G_w&= \begin{pmatrix}
   0& 1\\ 1&0  \end{pmatrix} -\left(\tfrac{(n+2w)}{(p+w)(n-p+w)} + a_{21}\right)  \begin{pmatrix}     0& n-p\\ p& 0  \end{pmatrix}
 + c_{12}  \begin{pmatrix}  p(n-p) & 0\\0 & -w(w+n)
     \end{pmatrix}  -a_{22}\mathrm{Id},
     \\    \mbox{} \displaybreak[0] \\
    H_w&=\frac{w(w+n) }{(p+w)(n-1+2w)(n-p+w)}  \begin{pmatrix}        (p+w-1)(n-p+w+1)&0\\0&(p+w+1)(n-p+w-1)     \end{pmatrix}
    \\ & \qquad + c_{12}\frac{w(n+w)}{(n-1+2w)} \begin{pmatrix}
      0& \frac{p(n-p+w-1)}{p+w}\\ -\frac{(p+w-1)(n-p)}{(n-p+w)}& 0
    \end{pmatrix}.
      \end{align*}
\end{prop}

\begin{proof}
  We look for constant matrices  $F_w, \tilde F_w, G_w, \tilde G_w, H_w$ and $\tilde H_w$ such that
  \begin{equation}\label{diff-prop}
  \begin{split}
  (1-x^2)\frac {d R_w}{dx}(x) &=  x F_w R_w(x) +R_w(x) \tilde F_w +G_w R_w(x)
  + R_w(x) \widetilde G_w  
  + H_wR_{w-1}(x) 
  \end{split}
  \end{equation}
  for all $w\in \NN_0$.
The polynomial  $R_w$ is of the form 
$$ R_w= \sum_{w=0}^w R_j^{w} x^j, \qquad \text {with } R^w_w=\mathrm I.$$
First of all, we can see that $-w=F_w+\tilde F_w$. Moreover,
 we have that \eqref{diff-prop} holds if and only if
\begin{equation}\label{verif}
  \begin{split}
    (j+1) R^w_{j+1}-(j-1) R_{j-1}^w & = -w  R_{j-1}^w + F_w R_{j-1}^w- R_{j-1}^w F_w + G_w R_{j}^w
    + R_{j}^w \tilde G_w + H_w R^{w-1}_j.
  \end{split}
\end{equation}
 The coefficients $R^w_j$ were computed explicitly in \cite{Z16}. Depending on whether $w-j$ is odd or even, we have different expressions:
\begin{align*}
  R^{w}_{w-2k}&=\frac{w! (-1)^k}{2^{2k} k! \,(w-2k)!\,(\tfrac{n+1}2+w-k)_k } \begin{pmatrix}
    \frac{p+w-2k}{p+w}&0\\ 0& \frac{n-p+w-2k}{n-p+w}
  \end{pmatrix},& & 0\le k \le \left \lfloor{w/2}\right \rfloor, \\
  R^{w}_{w-2k-1}&=\frac{w! (-1)^k}{2^{2k} k! \,(w-2k-1)!\,(\tfrac{n+1}2+w-k)_k } \begin{pmatrix}
    0& \frac{1}{p+w}\\  \frac{1}{n-p+w}
  \end{pmatrix},& & 0\le k \le \left \lfloor{(w-1)/2}\right \rfloor.
\end{align*}

The equation \eqref{verif} is equivalent to verifying the following identities, for every integer $k\geq 0$,
\begin{align*}
  &(w-2k+1)  R^w_{w-2k+1}-(w-2k-1)R^w_{w-2k-1} \\ & \qquad  =  -w R^w_{w-2k-1}+ F_w R^w_{w-2k-1}  -R^w_{w-2k-1}F_w
  + G_w R^w_{w-2k}+ R^w_{w-2k} \tilde G_w + H_w R^{w-1}_{w-2k} ,\\
  & \mbox{} \\
  &(w-2k) R^w_{w-2k}-(w-2k-2)R^w_{w-2k-2}  \\ & \qquad = -w R^w_{w-2k-2}+ F_w R^w_{w-2k-2}  -R^w_{w-2k-2}F_w
  + G_w R^w_{w-2k-1}+ R^w_{w-2k-1} \tilde G_w + H_w R^{w-1}_{w-2k-1}.
\end{align*}

\noindent Now,  the proposition follows from straightforward computations.
\end{proof}

By combining different cases in Proposition \ref{diff} we obtain  the following useful results:

\begin{cor}
  The monic orthogonal polynomials $\{R_w\}$ satisfy
  $$ 0=x\big( M_w R_w(x)-R_w(x) M_w\big) +N_w R_w(x)+R_w(x) \widetilde N_w ,$$
  where
  \begin{align*}
    M_w&= \begin{pmatrix}      p&0\\0&n-p    \end{pmatrix},\quad
    N_w=-\begin{pmatrix}
      0 & \frac{p(n-p+w)}{p+w}\\ \frac{(n-p)(p+w)}{n-p+w}    \end{pmatrix}, \quad
    \widetilde N_w=\begin{pmatrix}
      0&n-p\\p&0     \end{pmatrix}.
  \end{align*}
\end{cor}

\begin{proof}
It follows from Proposition \ref{diff}: combining the result for the values $a_{21}=1, c_{12}=0, a_{11}=0, a_{22}=0$ with the same result for the values
$a_{21}=2, c_{12}=0, a_{11}=0, a_{22}=0$.
\end{proof}

\begin{cor}
  The orthogonal polynomials $\{R_w\}$ satisfy
  $$ 0=x\big( M_w R_w(x)-R_w(x) M_w\big) +N_w R_w(x)+R_w(x) \widetilde N_w +J_w R_{w-1},$$
  where
  \begin{align*}
    M_w&= \frac{(p+w)(n-p+w)}{n-2p}\begin{pmatrix}      0&p\\n-p& 0    \end{pmatrix}\, , \quad
    N_w=\big(w(w+n)-p(n-p)\big) \begin{pmatrix}
          1&0\\0&0         \end{pmatrix}  \\
    \widetilde N_w& =\begin{pmatrix}
       p(n-p) & 0\\0 & -w(w+n)
     \end{pmatrix} ,  \quad
    J_w= \frac{w(n+w)}{(n-1+2w)} \begin{pmatrix}
      0& \frac{p(n-p+w-1)}{p+w}\\ -\frac{(p+w-1)(n-p)}{(n-p+w)}& 0
    \end{pmatrix}.
  \end{align*}
\end{cor}
\begin{proof}
It follows from Proposition \ref{diff}: combining the result for the values $a_{21}=0, c_{12}=2, a_{11}=0, a_{22}=0$ with the same result for the values
$a_{21}=0, c_{12}=1, a_{11}=0, a_{22}=0$.
\end{proof}

\

Starting with the differentiation formulas for the monic orthogonal polynomials $R_w$ we obtain formulas corresponding to the orthonormal sequence $Q_w$.

\begin{prop}\label{diffON}
 Let $\{Q_w\} $ be an orthonormal sequence of matrix polynomials associated to the weight $W=W_{p,n}$ introduced in \eqref{peso-x}. 
We have
  \begin{align*}
  (1-x^2)\frac {d Q_w}{dx}(x)= & -w\, x  Q_w(x)+ x\big( \overline F_w Q_w(x)-Q_w(x) F_w\big) +\overline G_w Q_w(x)\\
  &+ Q_w(x) \widetilde G_w + \overline H_w Q_{w-1}(x),
  \end{align*}
  with
  \begin{align*}
    \overline{F}_w=\|R_w\|^{-1} F_w \|R_w\|,&&
  \overline{G}_w=\|R_w\|^{-1} G_w \|R_w\|,&&
  \overline{H}_w=\|R_w\|^{-1} H_w \|R_{w-1}\|,
    \end{align*}
  where $F_w$, $G_w$, $H_w$  and $\widetilde G_w$ are those given in Proposition \ref{diff}.
\end{prop}

\begin{proof}
The statement follows by multiplying the identity given in  Proposition \ref{diff} by $\|R_w\|^{-1}$ and from the fact that $Q_w=\|R_w\|^{-1} R_w$.
\end{proof}

\

\section{Conclusion and outlook}\label{CyO}

The main result derived in the previous sections is the existence of an explicit differential operator $\tilde D$  which, as we proved, commutes with $S$.

If one compares this result with the one in the celebrated series of papers by
D. Slepian, H. Landau and H. Pollak one may say that we are at the stage of their
first papers. What is needed now is an argument to conclude that the eigenfunctions of $\tilde D $ will automatically be eigenfunctions of the integral operator $S$.

In the series of papers mentioned above the simplicity of the spectrum of $\tilde D$ follows from classical Sturm-Liouville theory and this guarantees that they
have found a good way to compute the eigenvectors of $S$.

In our situation, things could eventually be reduced to that case, but in principle $\tilde D$, as well as $S$, have ``matrix valued eigenvalues,'' and the appropriate notion of ``simple spectrum" requires careful handling. For a recent
careful analysis of the spectral problem in the scalar case, see \cite{K16,KM12}.

At this point this appears as a non-trivial project and we intend to develop it in a future publication. Part of this  project is to develop numerical tools to compute the matrix valued eigenfunctions of $\tilde D$ most likely using expansions of analogs of the Legendre polynomials used in \cite{ORX13}.

There are, of course,   a number of different self-adjoint extensions of our symmetric differential operator $\tilde D$. We are in the ``limit circle" situation of H. Weyl at both endpoints. Only one extension   is   of interest to us, this issue already appears in the scalar case, and it has been carefully discussed in \cite{K16,KM12}. In our matrix value setup this will be part of the project mentioned above.

\bibliographystyle{alpha}

\newcommand{\etalchar}[1]{$^{#1}$}

\end{document}